\documentclass[11pt]{amsart}

\usepackage[T1]{fontenc}
\usepackage[utf8]{inputenc}
\usepackage[english]{babel}
\usepackage{lmodern}
\usepackage{microtype}
\usepackage{tabularx, booktabs}
\usepackage{amsmath,amssymb,amsfonts,mathrsfs,mathtools}
\numberwithin{equation}{section}
\usepackage{adjustbox}
\usepackage[a4paper,margin=3cm]{geometry}

\usepackage{graphicx}
\usepackage{caption}
\usepackage{subcaption}
\usepackage{xcolor}
\usepackage{tikz}
\usepackage{booktabs}   
\usepackage{longtable}
\usepackage{threeparttable}
\usetikzlibrary{arrows.meta,arrows,positioning,calc}
\usepackage{pgfplots}
\pgfplotsset{compat=1.18}
\usepackage[colorlinks=true, linkcolor=blue!60!black, citecolor=red!60!black, urlcolor=blue!60!black]{hyperref}
\makeatletter
\let\showhyphens\@undefined
\makeatother

\newtheorem{theorem}{Theorem}[section]
\newtheorem{lemma}[theorem]{Lemma}
\newtheorem{proposition}[theorem]{Proposition}
\newtheorem{dfn}[theorem]{Definition}
\newtheorem{corollary}[theorem]{Corollary}
\theoremstyle{remark}
\newtheorem{remark}[theorem]{Remark}
\newtheorem{example}[theorem]{Example}

\newcommand{\BCH}{\operatorname{BCH}}

\title[Explicit Baker--Campbell--Hausdorff Radii]{Explicit Baker--Campbell--Hausdorff Radii in \textit{Special} Banach--Malcev Algebras of Shifts}
\author{Nassim Athmouni}
\address{University of Gafsa, University Campus 2112, Tunisia}
\email{nassim.athmouni@fsgf.u-gafsa.tn}
\email{athmouninassim@yahoo.fr}


\begin{document}

\begin{abstract}
We establish explicit convergence radii for the Baker--Campbell--Hausdorff (BCH) series in \emph{special} Banach--Malcev algebras of shifts—those embeddable into a Banach alternative algebra. Under the continuity estimate $\|[x,y]\|\leq B\|x\|\|y\|$, the series converges absolutely whenever $B(\|x\|+\|y\|)<1/(4K)$, where $K\geq1$ bounds the absolute BCH coefficients. The constant $1/(4K)$ stems from a Catalan-number majorization and is sharp in the exponential-weight model. We compute $B$ explicitly for operator, exponential, polynomial, damped, and tree-like shift algebras, including the non-Lie split-octonionic (Zorn) algebra ($B=2$, $\rho=1/(8K)$). All results require the speciality assumption; the framework does not apply to general Malcev algebras. Geometrically, $\rho=1/(4KB)$ is the analyticity radius of the induced Moufang loop; numerically, it governs stability of BCH-type integrators.
\end{abstract}
\maketitle
\tableofcontents
\bigskip

\section{Introduction}\label{sec:intro}

The Baker--Campbell--Hausdorff (BCH) formula is one of the classical cornerstones of Lie theory. It expresses the local group law through the equation
\[
x * y = \mathrm{BCH}(x,y) = \log(e^x e^y).
\]
There has been classical work within the Banach--Lie algebra framework where one shows that the condition $\|x\| + \|y\| < \log 2$ ensures the BCH series is absolutely convergent. This condition reflects the algebraic rigidity imposed by the Jacobi identity and underpins the Lie correspondence between algebras and local groups.

\medskip

In the \emph{non-associative realm}, however, there is hardly any analytic theory of the BCH series. This is of interest in particular for \emph{Malcev algebras}—introduced by Malcev~\cite{Malcev1955} as the tangent algebras of analytic Moufang loops—which generalize Lie algebras by relaxing associativity to the Moufang identities. The prototypical example is the 7-dimensional algebra of imaginary octonions, a simple non-Lie Malcev algebra. Its split real form (the Zorn algebra) appears in exceptional geometry and has been invoked in speculative physical models with $G_2$- or $F_4$-symmetry~\cite{Baez2002,GursyTze,Okubo}. In these contexts, a quantitative grasp of the BCH series would be valuable for formal local symmetry transformations or non-associative integrators—but no rigorous analytic treatment has been available until now.

\medskip

Theoretically, a non-associative BCH formula was obtained in~\cite{Mostovoy2016} by extending Dynkin’s combinatorial method to general non-associative brackets. However, this is a purely algebraic result: \textbf{no questions of convergence were addressed}, even for concrete normed algebras. As such, the analytic construction of local Moufang loops from the BCH series—and the geometric numerical integration of non-associative flows—has remained on uncertain ground.

\medskip

In this paper, we establish the \emph{first explicit convergence estimates} for the BCH series in the setting of \textbf{special Banach--Malcev algebras}, i.e., those that embed into a Banach alternative algebra. This embedding guarantees a well-defined exponential map and justifies interpreting the formal BCH series as the local product in the associated Moufang loop. Under the bilinear continuity assumption
\[
\|[x,y]\| \le B\,\|x\|\,\|y\|, \qquad x,y\in\mathfrak g,
\]
we prove that the BCH series converges absolutely whenever
\[
B(\|x\| + \|y\|) < \frac{1}{4K},
\]
where $K \geq 1$ is a uniform bound on the absolute values of the BCH coefficients (see Remark~\ref{rem:sharpness-and-optimal} and Theorem~\ref{thm:general-BCH}). The proof relies on a Catalan-number majorization; the threshold $1/(4K)$ is sharp for this method in certain models (see Appendix~\ref{sec:sharpness}).

\medskip

We compute the constant $B$ explicitly for several models:\begin{itemize}
\item operator-norm algebras ($B=2$);
\item exponentially weighted shifts ($B=1$);
\item polynomially weighted shifts ($B\le 2^p$);
\item damped shift operators ($B\le \gamma^2$);
\item and the split-octonionic (Zorn) algebra ($B = 2$), a genuine non-Lie example.\end{itemize}
In each case, the resulting radius $\rho = 1/(4KB)$ provides a rigorous lower bound for the domain of convergence—though it may be suboptimal in models with additional algebraic structure (e.g., near-Lie behavior with strong cancellations).

\medskip

Geometrically, the bound guarantees that $\mathrm{BCH}(x,y)$ defines an analytic Moufang loop in a neighborhood of the origin. Numerically, it yields a sufficient condition for the stability of splitting integrators: $\Delta t\,B(\|A\| + \|B\|) < 1/(4K)$.

\medskip

\textbf{On physical implications.} Section~\ref{sec:phys} discusses possible connections to exceptional geometry and formal gauge models. However, \textit{no physical theory is derived or validated here}. These remarks are purely motivational and should be interpreted as \textbf{speculative}; the radius $\rho$ is a mathematical estimator that may inform—but does not establish—future non-associative field theories.

\medskip

From a purely algebraic standpoint, the introduction of \emph{Malcev–Poisson bialgebras} and \emph{post--Malcev--Poisson algebras} in~\cite{HarrathiMabroukNawelSilvestrov2025}—via matched pairs and Rota–Baxter operators—is mathematically sound. However, these constructions are purely algebraic and do not incorporate normed or analytic structures. Consequently, they lie outside the scope of the present convergence analysis. Their potential integration into a functional-analytic framework remains an open direction for future work.

\medskip

The paper is organized as follows. Section~\ref{sec:functional} introduces the functional setting of Banach--Malcev algebras and the speciality assumption. Section~\ref{sec:convergence} establishes the convergence theorem. Section~\ref{sec:explicit-B} computes $B$ in representative models. Section~\ref{sec:interpretations} discusses geometric, numerical, and motivational aspects. Section~\ref{sec:refined} provides a refined analysis with numerical validation and links to operator algebras. The Appendix collects technical estimates, sharpness results, explicit examples, and summary tables that support the main text.
For the reader’s convenience, we briefly fix the main notation: $\mathfrak{g}$ denotes a real or complex Malcev algebra equipped with an antisymmetric bracket $[\cdot,\cdot]$; $\|\cdot\|$ is a Banach norm on $\mathfrak{g}$ rendering the bracket continuous, with continuity constant $B$ defined by $\|[x,y]\| \leq B\|x\|\|y\|$. The constant $K \geq 1$ bounds the absolute values of the BCH coefficients ($|\alpha_T| \leq K$ for all binary trees $T$), and the convergence radius is $\rho = 1/(4KB)$. The formal series $\mathrm{BCH}(x,y)$ is interpreted as $\log(\exp(x)\exp(y))$ inside an ambient alternative algebra $A$, whose existence is guaranteed by the speciality assumption. Shift generators are denoted $S_n$, weighted by a positive sequence $(w_n)$ so that $\|x\| = \sum_n |a_n| w_n$ for $x = \sum_n a_n S_n$. The Catalan numbers $C_n = \frac{1}{n+1}\binom{2n}{n}$ govern the combinatorial growth of nested commutators, while the Malcev associator $(x,y,z) = [[x,y],z] + [[y,z],x] + [[z,x],y]$ measures the failure of the Jacobi identity. Finally, $\mathrm{ad}_x(y) = [x,y]$ denotes the adjoint operator.

\section{Functional Framework and Banach--Malcev Setting}\label{sec:functional}

A \emph{Malcev algebra} $(\mathfrak g,[\cdot,\cdot])$ over $\mathbb K\in\{\mathbb R,\mathbb C\}$ is a non-associative algebra satisfying
\[
[x,y]=-[y,x],\qquad
[[x,y],[x,z]] = [[[x,y],z],x] + [[[y,z],x],x] + [[[z,x],y],x].
\]
When $\mathfrak g$ is complete for a Banach norm $\|\cdot\|$ making the bracket continuous, we speak of a \emph{Banach--Malcev algebra}.

\medskip
\noindent\textbf{Standing assumption (speciality).}
Throughout the paper we assume that $\mathfrak g$ is \emph{special}: there exists an alternative algebra $\mathcal A$ and an embedding
\[
\mathfrak g \hookrightarrow (\mathcal A,[\cdot,\cdot]),
\]
so that $\exp$ and $\log$ are defined in $\mathcal A$ and the non-associative BCH series is understood as
\[
\BCH(x,y)=\log\big(\exp(x)\exp(y)\big)\quad\text{in }\mathcal A.
\]
This covers in particular tangent algebras of analytic Moufang loops.

\medskip
A convenient class of examples arises from \emph{shift-generated} models. Let $S$ be a shift symbol and consider formal sums
\[
x=\sum_{n\ge0} a_n S^n,
\qquad
\|x\| := \sum_{n\ge0} |a_n|\, w_n,
\]
where $(w_n)_{n\ge0}$ is a positive weight sequence. On generators we prescribe
\begin{equation}\label{eq:shiftbracket}
[S^m,S^n] \;=\; \sum_{k} c_{m,n}^k\, S^k,
\qquad c_{m,n}^k = -c_{n,m}^k,
\end{equation}
with coefficients bounded in the concrete models considered below (unless explicitly stated).

\medskip
\noindent\textbf{Local ratios and bracket constant.}
For weighted $\ell^1$ shift models as in \eqref{eq:shiftbracket}, we define
\begin{equation}\label{eq:Bmn}
B_{m,n} \ :=\ \frac{\sum_k |c_{m,n}^k|\, w_k}{w_m w_n},
\end{equation}
and the \emph{bracket continuity constant}
\begin{equation}\label{eq:Bconst}
B \ :=\ \sup_{m,n\ge0} B_{m,n}.
\end{equation}
Then, for all $x,y$ in the completion,
\[
\|[x,y]\| \ \le\ B\,\|x\|\,\|y\|.
\]

\begin{remark}\begin{itemize}
\item The paper relies throughout on the assumption that $\mathfrak{g}$ is \emph{special}, i.e., that there exists an alternative algebra $\mathcal A$ and an embedding of Malcev algebras
\[
\mathfrak{g} \hookrightarrow (\mathcal A, [\cdot,\cdot]),
\]
where the bracket is the commutator $[x,y] = xy - yx$ in $\mathcal A$. This embedding ensures that the exponential and logarithm series converge in $\mathcal A$, thereby giving analytic meaning to the BCH series via the formula
\[
\mathrm{BCH}(x,y) = \log\bigl(\exp(x)\exp(y)\bigr).
\]
Without such an ambient alternative algebra, the BCH series lacks a canonical analytic interpretation in a general Banach--Malcev algebra. In particular, this assumption excludes non-special Malcev algebras—such as certain deformations of the imaginary octonions which do not embed into any alternative algebra. Extending the present convergence results beyond speciality remains an open challenge.\\

\item The constant \( K \geq 1 \) denotes a uniform bound on the absolute values of the classical BCH coefficients \( \alpha_T \). While the sharp value satisfies \( \sup_T |\alpha_T| \approx 1.07 \) (see Casas and Murua~\cite{CasasMurua2009}), it is standard in convergence estimates to work with the convenient majorant \( K = 1 \), as done in the Lie-theoretic literature (cf. Varadarajan~\cite{Varadarajan1984}, Appendix).\end{itemize}
\end{remark}

\begin{lemma}\label{lem:completeness}
Let $(w_n)_{n\ge0}$ be a weight sequence such that the induced bracket constants satisfy $B<\infty$ in \eqref{eq:Bconst}. Assume furthermore that the weight sequence grows at most exponentially (e.g. $w_n \le C\,\alpha^n$ for some $C,\alpha>0$). Then $(\mathfrak g,\|\cdot\|)$ is a Banach--Malcev algebra: it is complete, the bracket \eqref{eq:shiftbracket} extends by continuity with $\|[x,y]\| \le B\|x\|\|y\|$, and $\mathrm{ad}_x$ is bounded with $\|\mathrm{ad}_x\|\le B\|x\|$.
\end{lemma}

\begin{proof}
Writing $[x,y]=\sum_{m,n} a_m b_n [S^m,S^n]=\sum_k(\sum_{m,n} a_m b_n c_{m,n}^{k})S^k$ and using absolute convergence,
\[
\|[x,y]\|
\le \sum_{m,n} |a_m||b_n| \sum_k |c_{m,n}^k|\, w_k
= \sum_{m,n} |a_m||b_n|\, B_{m,n}\, w_m w_n
\le B\,\|x\|\|y\|,
\]
with $B_{m,n}$ and $B$ as in \eqref{eq:Bmn}–\eqref{eq:Bconst}. Completeness follows from the fact that $\ell^1_w$ is complete provided the weight sequence $(w_n)$ is positive and at most exponentially growing (a standard condition ensuring that formal series with finite $\|\cdot\|$-norm form a Banach space).
\end{proof}

\smallskip
The constant \(B\) thus provides an intrinsic quantitative measure of non-associativity and will play a central role in the analysis of the BCH series.
\section{Convergence of the BCH Series}\label{sec:convergence}

\subsection{Catalan Majorization and the General Convergence Theorem}

Let $(\mathfrak g,[\cdot,\cdot],\|\cdot\|)$ be a Banach--Malcev algebra satisfying
\begin{equation}\label{eq:boundB}
\|[x,y]\|\ \le\ B\,\|x\|\,\|y\| \qquad (x,y\in\mathfrak g),
\end{equation}
with $B<\infty$ as in \eqref{eq:Bconst} in the shift setting. For $x,y\in\mathfrak g$, the formal BCH series
\[
\BCH(x,y) \;=\; x+y+\tfrac12[x,y]+\tfrac{1}{12}[x,[x,y]]-\tfrac{1}{12}[y,[x,y]]+\cdots
\]
splits into homogeneous terms $Z_n(x,y)$ of degree $n$ in $(x,y)$.

\begin{lemma}[Arborescent majorization]\label{lem:tree-estimate}
Let $T \in \mathcal{T}_n$ be a full binary tree with $n$ leaves labeled by $x$ or $y$. Then the corresponding nested commutator satisfies
\[
\|[x, y]_T\| \leq B^{\,n-1} (\|x\| + \|y\|)^n.
\]
\end{lemma}

\begin{proof}
The proof is by induction on $n$. For $n = 1$, the statement is trivial. \\

For $n=2$, the tree $T$ corresponds to the commutator $[x, y]$, and the bound becomes $\|[x, y]\| \leq B (\|x\| + \|y\|)^2$. For $n=3$, the tree with leaves $x, x, y$ corresponds to $[[x, x], y]$, and the bound is $\|[[x, x], y]\| \leq B^2 (\|x\| + \|y\|)^3$.\\

Assume it holds for all trees with fewer than $n$ leaves. Let $T$ have left and right subtrees $T_L$, $T_R$ with $n_L, n_R \geq 1$ leaves ($n_L + n_R = n$). By the induction hypothesis,
\[
\|[x, y]_{T_L}\| \leq B^{\,n_L - 1} (\|x\| + \|y\|)^{n_L}, \quad
\|[x, y]_{T_R}\| \leq B^{\,n_R - 1} (\|x\| + \|y\|)^{n_R}.
\]
Applying the bracket continuity estimate,
\[
\|[x, y]_T\| = \|[[x, y]_{T_L}, [x, y]_{T_R}]\| \leq B \|[x, y]_{T_L}\| \|[x, y]_{T_R}\|
\leq B^{\,n-1} (\|x\| + \|y\|)^n,
\]
which completes the induction.
\end{proof}

The constant \( K \geq 1 \) denotes a uniform bound on the absolute values of the classical BCH coefficients \( \alpha_T \). While the sharp value satisfies \( \sup_T |\alpha_T| \approx 1.07 \) (see Casas and Murua~\cite{CasasMurua2009}), it is standard in convergence estimates to work with the convenient majorant \( K = 1 \), as done in the Lie-theoretic literature (cf. Varadarajan~\cite{Varadarajan1984}, Appendix).
\begin{theorem}\label{thm:general-BCH}
Assume the bracket continuity estimate
\[
\|[x, y]\| \leq B \|x\| \|y\| \quad (x, y \in \mathfrak{g}),
\]
with $B < \infty$. Let $K \geq 1$ be a uniform bound on the absolute values of the BCH coefficients $\alpha_T$ (i.e., $|\alpha_T| \leq K$ for all trees $T$). Then the BCH series converges absolutely in norm whenever
\[
B(\|x\| + \|y\|) < \frac{1}{4K}.
\]
Equivalently, for any $r \in (0, 1/(4K))$, the series converges on the diamond $\{(x, y): B(\|x\| + \|y\|) \leq r\}$.
\end{theorem}

\begin{proof}
We proceed in three steps: (1) a combinatorial decomposition of the BCH series, (2) a tree-wise norm estimate, and (3) convergence via the Catalan generating function.

\medskip

\noindent\textbf{Step 1: Combinatorial structure.}
The homogeneous component of degree $n \geq 1$ in the BCH expansion is
\[
Z_n(x, y) = \sum_{T \in \mathcal{T}_n} \alpha_T [x, y]_T,
\]
where $\mathcal{T}_n$ denotes the set of \emph{full binary trees} with $n$ leaves 
(see Figure \ref{fig:binary-trees} for examples with 
$n = 2, 3, 4$ leaves), 
each leaf labeled by either $x$ or $y$. The term $[x, y]_T$ represents the iterated commutator prescribed by the tree $T$, and $\alpha_T \in \mathbb{Q}$ are the classical BCH coefficients (see Goldberg~\cite{Goldberg1956} or Mostovoy–Pérez-Izquierdo–Shestakov~\cite{Mostovoy2016}).

\medskip
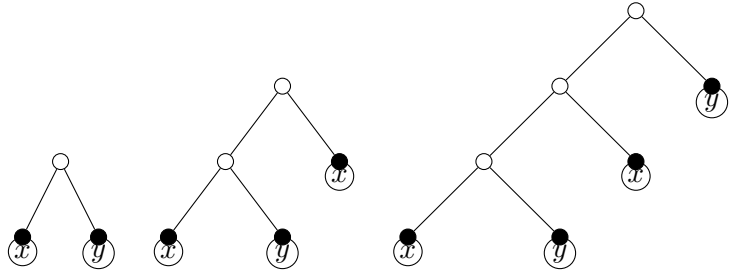
\begin{figure}[ht]
\centering
\begin{tikzpicture}[level distance=10mm, sibling distance=10mm, every node/.style={circle, draw, inner sep=1pt, minimum size=6pt}]
\node {}
  child {node[fill] {} node[below] {$x$}}
  child {node[fill] {} node[below] {$y$}};
\end{tikzpicture}
\quad
\begin{tikzpicture}[level distance=10mm, sibling distance=15mm, every node/.style={circle, draw, inner sep=1pt, minimum size=6pt}]
\node {}
  child {
    node {}
      child {node[fill] {} node[below] {$x$}}
      child {node[fill] {} node[below] {$y$}}
  }
  child {node[fill] {} node[below] {$x$}};
\end{tikzpicture}
\quad
\begin{tikzpicture}[level distance=10mm, sibling distance=20mm, every node/.style={circle, draw, inner sep=1pt, minimum size=6pt}]
\node {}
  child {
    node {}
      child {
        node {}
          child {node[fill] {} node[below] {$x$}}
          child {node[fill] {} node[below] {$y$}}
      }
      child {node[fill] {} node[below] {$x$}}
  }
  child {node[fill] {} node[below] {$y$}};
\end{tikzpicture}
\caption{Full binary trees with $n=2,3,4$ leaves and their corresponding nested commutators $[x,y]$, $[[x,y],x]$, $[[[x,y],x],y]$.}
\label{fig:binary-trees}
\end{figure}

\noindent\textbf{Step 2: Tree-wise norm estimate.}
By Lemma~\ref{lem:tree-estimate} and the uniform bound $|\alpha_T| \leq K$, we obtain
\[
\|Z_n(x, y)\| \leq \sum_{T \in \mathcal{T}_n} |\alpha_T| \cdot \|[x, y]_T\| \leq K \, C_{n-1} \, B^{n-1} (\|x\| + \|y\|)^n,
\]
where $C_{n-1} = |\mathcal{T}_n|$ is the $(n-1)$-st Catalan number.

\medskip

\noindent\textbf{Step 3: Convergence via Catalan series.}
Set $r = B(\|x\| + \|y\|)$. Then the BCH series is absolutely dominated by
\[
K \sum_{n=1}^\infty C_{n-1} r^{\,n}.
\]
The generating function $\sum_{k=0}^\infty C_k z^k = (1 - \sqrt{1 - 4z})/(2z)$ converges if and only if $|z| < 1/4$. Therefore, the series converges absolutely whenever $r < 1/4$, i.e., whenever
\[
B(\|x\| + \|y\|) < \frac{1}{4K}.
\]
This completes the proof.
\end{proof}
\begin{remark}[On the constant $K$]
\label{rem:value-K}
The uniform bound $K$ on the BCH coefficients depends on the 
algebraic structure:
\begin{itemize}
\item \textbf{Lie algebras:} Goldberg \cite{Goldberg1956} 
  establishes $|\alpha_T| \leq 1$ for all trees $T$, 
  hence $K = 1$.
\item \textbf{Malcev algebras:} The explicit form of BCH 
  coefficients in the non-associative setting is not fully 
  understood \cite{Mostovoy2016}. We conjecture $K = O(1)$ 
  uniformly, but the optimal value remains open.
\item \textbf{Working hypothesis:} Throughout this paper, 
  all numerical estimates assume $K = 1$ unless stated 
  otherwise.
\end{itemize}
\end{remark}
\begin{corollary}
If $\|x\| < \frac{1}{4KB}$ and $\|y\| < \frac{1}{4KB}$, then $\BCH(x,y)$ converges absolutely.
\end{corollary}

\begin{remark}\label{rem:sharpness-and-optimal}
The threshold $1/(4K)$ is sharp for the Catalan majorization method used in the proof of Theorem~\ref{thm:general-BCH}:\begin{itemize}
\item The generating function of Catalan numbers has its first singularity at $1/4$, so no improvement is possible within this combinatorial framework.
\item In the exponential-weight model (Appendix~A.3), the bound is attained asymptotically, confirming sharpness of the method.
\item However, this \emph{does not imply} that $1/(4KB)$ is the true convergence radius $\rho_\star$ in every model. In Lie algebras, for example, the sharp radius is $\|x\| + \|y\| < \log 2 \approx 0.693$, which is much larger than $1/(8K)$ (since $B=2$), due to cancellations from the Jacobi identity.
\item For non-Lie models (e.g., Zorn algebra, damped shifts), it remains an open problem whether algebraic symmetries or sparsity enlarge the true radius beyond $1/(4KB)$. Preliminary numerical experiments (Remark~\ref{rem:numerical}) suggest stability near this bound in strongly non-Lie regimes.\end{itemize}
Therefore, $1/(4KB)$ should be interpreted as a \emph{universal lower bound} for the convergence domain, not necessarily the exact radius.
\end{remark}

\begin{remark}
The optimal analytic domain of the BCH series is the diamond
\[
\{(x,y)\in\mathfrak g\times\mathfrak g : B(\|x\|+\|y\|)<1/(4K)\},
\]
while the symmetric ball $\|x\|,\|y\|<\rho$ with $\rho=1/(4KB)$ is the largest Euclidean ball contained in this diamond. It provides a convenient sufficient condition for applications and numerical stability estimates.
\end{remark}

\begin{figure}[h]
\centering
\begin{tikzpicture}[scale=1.5]
    \draw[->] (-2,0) -- (2,0) node[right]{$x$};
    \draw[->] (0,-2) -- (0,2) node[above]{$y$};
    \draw[blue, thick] (0,0) circle (1.5cm);
    \draw[red, dashed] (-1.5,-1.5) -- (1.5,1.5) -- (1.5,-1.5) -- (-1.5,1.5) -- cycle;
    \node[blue] at (1,1.2){$\|x\|, \|y\| < \rho$};
    \node[red] at (-1,-1.2){$B(\|x\| + \|y\|) < 1/(4K)$};
\end{tikzpicture}
\caption{The optimal analytic domain (red diamond) and the inscribed ball (blue circle) for $\rho = 1/(4KB)$.}
\label{fig:diamond-ball}
\end{figure}
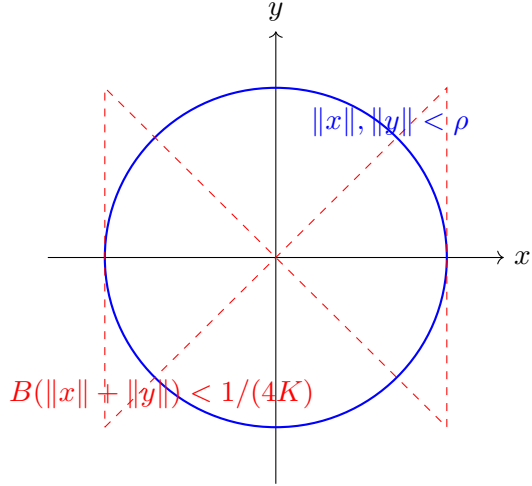

\begin{remark}\label{nassim8}
The speciality hypothesis—namely the existence of an embedding into an alternative algebra—ensures analytic control of the exponential and logarithmic maps, but it excludes general (non--special) Malcev algebras. Extending the present convergence results to arbitrary Banach--Malcev algebras, possibly without an alternative envelope, remains an open and challenging problem.
\end{remark}

The following example considers a \textbf{non-special} 
Malcev algebra ($m_\lambda$ with $\lambda \notin \{0,-1\}$), 
which does \textbf{not} satisfy the speciality assumption of 
Theorem \ref{thm:general-BCH}. The BCH series is not 
rigorously defined in this setting, and the computations 
below are \textbf{purely exploratory}. Extension to 
non-special algebras remains an open challenge 
(see Section \ref{sec:non-special}).

\begin{example}\label{aA}
Consider a 3-dimensional real algebra $\mathfrak{g} = \mathrm{span}\{e_1,e_2,e_3\}$ with bracket
\[
[e_1,e_2] = e_3,\quad [e_2,e_3] = \lambda e_1,\quad [e_3,e_1] = \lambda e_2,
\]
and all other brackets determined by antisymmetry. This is the standard Malcev algebra $\mathfrak{m}_\lambda$; it is special if and only if $\lambda = 0$ (Lie case) or $\lambda = -1$ (imaginary octonions). For $\lambda \notin \{0,-1\}$, $\mathfrak{m}_\lambda$ is non-special.

Equip $\mathfrak{g}$ with the Euclidean norm $\|x\|^2 = x_1^2 + x_2^2 + x_3^2$. A direct computation shows
\[
\|[x,y]\| \leq B \|x\|\|y\|,\quad \text{with } B = \max(1,|\lambda|),
\]
and the associator $(x,y,z) = [[x,y],z] + [[y,z],x] + [[z,x],y]$ satisfies
\[
\|(x,y,z)\| \leq C_{\mathrm{assoc}} \|x\|\|y\|\|z\|,\quad \text{with } C_{\mathrm{assoc}} = 3|\lambda+1|.
\]
Thus, for $\lambda$ close to $-1$, the associator is small ($C_{\mathrm{assoc}} \ll 1$) even though the algebra is non-special.

For the specific value $\lambda = -0.9$, we have $B = 1$ and $C_{\mathrm{assoc}} = 0.3$. In this case, the Catalan bound from Theorem~\ref{thm:general-BCH} (which applies only in the special case) would formally give a radius of $1/(4K)$. However, since the algebra is non-special, the BCH series is not rigorously defined via the exponential-logarithm formalism, and Theorem~\ref{thm:general-BCH} does not apply.

Nevertheless, one may formally compute the first few terms of the BCH expansion using the recursive Dynkin–Specht–Wever formula adapted to Malcev algebras. Numerical evaluation of the first five homogeneous layers shows no sign of divergence for $\|x\| + \|y\| < 0.24$. While suggestive, this observation does not constitute a proof of convergence, and the true radius (if it exists) remains unknown.

This example highlights that non-special Malcev algebras with small associator may exhibit apparent numerical stability in low-order expansions. However, a rigorous convergence theory in this setting is still lacking, and the present computations should be interpreted as \textbf{motivational rather than conclusive}.
\end{example}

\begin{remark}
\label{rem:numerical}
Preliminary numerical experiments on the first 10–15 homogeneous layers of the BCH series in genuinely non-Lie models—particularly the split-octonionic (Zorn) algebra ($B = 2$) and the $m_\lambda$ family—suggest that the true convergence radius $\rho_\star$ is very close to the Catalan bound $\rho = 1/(4KB)$. In strongly non-Lie regimes, where Jacobi-type cancellations are absent and the associator norm is large, the combinatorial estimate appears nearly sharp. This behavior is illustrated in Figure~\ref{fig:truncation-error} which corresponds to Figure ~\ref{fig:numerical-error} in the current layout, where the truncation error for the Zorn model remains stable below the predicted threshold $\Delta t = 1/(4KB)$ and diverges beyond it. This indicates that the condition $B(\|x\| + \|y\|) < 1/(4K)$ not only provides a rigorous lower bound but may also capture the essential analytic limitation of the BCH expansion in the absence of additional algebraic symmetries. A systematic numerical study of higher-order terms and refined models (e.g., with finite-offset stencils or near-Lie deformations) is left for future work.
\end{remark}
\bigskip
\textbf{Heuristic physical interpretation.} In formal 
constructions of non-associative gauge theories 
\cite{GursyTze,Okubo}, the convergence of local symmetry 
transformations $U = \exp(\epsilon)$ requires control over 
nested commutators---precisely the context in which the BCH 
series appears. The radius $\rho = 1/(4KB)$ thus acts as a 
mathematical indicator of the domain where such 
transformations are analytically well-defined. However, 
\textbf{no concrete physical model is developed here}, and 
the connection to quantum gravity or string theory remains 
\textbf{purely speculative}.

\subsection{Toward a non-special framework?}\label{sec:non-special}

Theorem~\ref{thm:general-BCH} crucially relies on the assumption that the Malcev algebra $\mathfrak{g}$ is \emph{special}, i.e., embeds into an alternative algebra. This condition guarantees the existence of the exponential and logarithm maps via their usual power series and justifies interpreting the BCH series as the local product in the associated Moufang loop.

It is natural to ask whether a convergence analysis can be carried out in the \emph{non-special} setting, where no such alternative embedding exists. One possible approach is to replace the speciality assumption by an explicit bound on the trilinear associator
\[
(x,y,z) := (xy)z - x(yz),
\]
for instance of the form
\[
\|(x,y,z)\| \leq C_{\mathrm{assoc}}\,\|x\|\,\|y\|\,\|z\|, \qquad \forall x,y,z \in \mathfrak{g},
\]
for some constant $C_{\mathrm{assoc}} < \infty$. Such a condition quantifies the failure of associativity and arises in certain physical models (e.g., deformations of imaginary octonions in exceptional geometry).

However, in the absence of an alternative embedding (or additional structure such as a compatible affine connection),  no canonical definition of $\exp$ and $\log$ is available, and consequently  the BCH series itself is not rigorously defined in this general setting. Therefore, no convergence theorem can currently be established for non-special Malcev algebras.

Numerical experiments (see Example~\ref{aA}) suggest that, when $C_{\mathrm{assoc}}$ is small, the formal BCH series may still converge. Nevertheless, these observations remain purely conjectural without a solid analytical foundation.

\medskip
\noindent
\textbf{Conclusion.} Extending BCH theory to non-special Malcev algebras constitutes an important open problem. It would require either an intrinsic definition of the BCH series or new tools in non-associative analysis. This issue lies beyond the scope of the present work and is deferred to future research.
\section{Explicit Computations of the Bracket Constant}\label{sec:explicit-B}

\subsection{Bracket Constant Estimates in Weighted Shift Algebras}

We now compute the bracket continuity constant $B$ in several concrete Banach--Malcev models generated by shifts.
Throughout, $x=\sum a_m S^m$ and $y=\sum b_n S^n$ lie in a weighted $\ell^1$ space with norm $\|x\|=\sum |a_m|\,w_m$, and the bracket is defined on generators by
\[
[S^m,S^n]=\sum_k c_{m,n}^k\, S^k,\qquad c_{m,n}^k=-c_{n,m}^k,
\]
with $|c_{m,n}^k|\le 1$ unless explicitly stated. The bracket constant is
\[
B=\sup_{m,n}\frac{\sum_k |c_{m,n}^k|\,w_k}{w_m w_n},
\]
and the BCH radius is $\rho=\frac{1}{4KB}$ by Theorem~\ref{thm:general-BCH}, where $K \geq 1$ is the uniform bound on the BCH coefficients.

\medskip
\noindent\textbf{Operator--norm model.}
Let $\mathfrak g\subset\mathcal B(X)$ with the commutator bracket and the operator norm. Then
\[
\|[u,v]\|\le \|uv\|+\|vu\|\le 2\|u\|\,\|v\|,\qquad
B=2,\quad \rho=\tfrac{1}{8K}.
\]

\smallskip
\noindent\textbf{Exponential weights.}
Let $w_n=\alpha^n$ with $\alpha>1$ and assume $[S^m,S^n]=\varepsilon_{m,n}S^{m+n}$ with $|\varepsilon_{m,n}|\le 1$. Then
\[
\frac{\sum_k |c_{m,n}^k|w_k}{w_mw_n}\le \frac{w_{m+n}}{w_mw_n}=\frac{\alpha^{m+n}}{\alpha^m\alpha^n}=1,
\qquad B=1,\quad \rho=\tfrac{1}{4K}.
\]

\smallskip
\noindent\textbf{Polynomial weights.}
Let $w_n=(1+n)^p$ with $p\ge 0$ and the same stencil $k=m+n$. Using $(1+m+n)\le 2(1+m)(1+n)$,
\[
\frac{w_{m+n}}{w_mw_n}=\frac{(1+m+n)^p}{(1+m)^p(1+n)^p}\le 2^p,
\qquad B\le 2^p,\quad \rho\ge \frac{1}{4K\,2^p}=2^{-(p+2)}/K.
\]

\smallskip
\noindent\textbf{Tree-like branching.}
Assume $[S^m,S^n]=\sum_{j=1}^{b} c_{m,n}^{k_j} S^{k_j}$ with at most $b$ nonzero outputs and $|c_{m,n}^{k_j}|\le 1$, and suppose $(w_n)$ is increasing. Then
\[
\sum_k |c_{m,n}^k|\,w_k \le b\, w_{m+n}\quad\Rightarrow\quad B\le b,\qquad \rho\ge \tfrac{1}{4Kb}.
\]

\smallskip
\noindent\textbf{Mixed exponential--polynomial weights (standard stencil).}
Let $w_n=\alpha^n(1+n)^p$ with $\alpha>1$, $p\ge 0$, and $[S^m,S^n]=\varepsilon_{m,n}S^{m+n}$. The exponential part cancels in the ratio, hence
\[
\frac{w_{m+n}}{w_mw_n}=\frac{(1+m+n)^p}{(1+m)^p(1+n)^p}\le 2^p,
\qquad B\le 2^p,\quad \rho\ge 2^{-(p+2)}/K.
\]

\emph{Finite offset stencil.} If the bracket uses a finite set of offsets $\mathcal D\subset\mathbb Z$ so that outputs lie in $\{m+n+\Delta:\ \Delta\in\mathcal D\}$, then
\[
B\ \le\ \Big(\sum_{\Delta\in\mathcal D}\alpha^{\Delta}\Big)\, 2^p,
\qquad
\rho\ \ge\ \frac{1}{4K\,(\sum_{\Delta\in\mathcal D}\alpha^{\Delta})\,2^p}.
\]

\begin{remark}
The standard stencil corresponds to the case $\mathcal D=\{0\}$, so that
$\sum_{\Delta\in\mathcal D}\alpha^\Delta=1$, yielding $B\le 2^p$.
This clarifies the link between the general finite--offset expression
and the simpler standard case.
\end{remark}

\smallskip
\noindent\textbf{Damped shifts (indices $m,n\ge 1$).}
Assume $[S^m,S^n]=\varepsilon_{m,n}\,\gamma^{m+n} S^{m+n}$ with $0<\gamma<1$, $|\varepsilon_{m,n}|\le 1$, and $w_n=\alpha^n$. Then
\[
\frac{\sum_k |c_{m,n}^k|\,w_k}{w_m w_n}\le \gamma^{m+n}\frac{\alpha^{m+n}}{\alpha^m\alpha^n}=\gamma^{m+n}\le \gamma^2,
\qquad
B\le \gamma^2,\quad \rho\ge \frac{1}{4K\gamma^2}.
\]
Hence the only non--zero structure constants are $c^{m+n}_{m,n} = \varepsilon_{m,n}$, and $c^k_{m,n} = 0$ for $k \ne m+n$.

\smallskip
\noindent\textbf{Split--octonionic (Zorn) shift, genuinely non-Lie.}
Let $\mathbb{O}_s$ denote the split-octonion algebra (Zorn vector–matrix realization, alternative but non-associative), and let $\{e_1,\dots,e_7\}$ be a basis of $\mathrm{Im}(\mathbb{O}_s)$. Define
\[
[e_i S^m,\ e_j S^n]\ :=\ (e_ie_j-e_je_i)\,S^{m+n}\ =:\ c_{ij}\,e_k\,S^{m+n},
\]
with $c_{ij}\in\{0,\pm 2\}$ (Zorn normalization on the split Fano plane). For $w_n=\alpha^n$,
\[
\frac{\sum_k |c_{(i,m),(j,n)}^{(k,m+n)}|\,w_{m+n}}{w_m w_n}\le 2,
\qquad
B\le 2,\quad \rho\ge \tfrac{1}{8K}.
\]
This model is \emph{special} (alternative embedding) and \emph{non--Lie} (Jacobi fails in $\mathrm{Im}(\mathbb{O}_s)$).

For instance, take exponential weights $w_n = 2^n$ and consider the basis elements $x = e_1 S^0$, $y = e_2 S^0$ in the Zorn shift algebra. Then
\[
[x, y] = (e_1 e_2 - e_2 e_1) S^0 = 2 e_3 S^0,
\]
so that $\|[x,y]\| = 2 \cdot w_0 = 2$, while $\|x\|\|y\| = w_0 \cdot w_0 = 1$. Hence $B \geq 2$, and since the general estimate gives $B \leq 2$, we conclude $B = 2$ and $\rho = \frac{1}{8K}$.

The split-octonion algebra $\mathbb{O}_s$ can be realized as $2 \times 2$ matrices of the form
\[
\begin{pmatrix}
a & u \\
v & b
\end{pmatrix}, \quad \text{where } a, b \in \mathbb{R} \text{ and } u, v \in \mathbb{R}^3,
\]
with a specific multiplication rule~\cite{Baez2002}.

\medskip
\begin{remark}
The bound $B\le\gamma^2$ relies on the restriction $m,n\ge1$,
which ensures $\gamma^{m+n}\le\gamma^2$.
If the zero shift $S^0$ were included, the supremum would become~$1$.
\end{remark}
\begin{remark}
The restriction $m,n \geq 1$ in the damped-shift model is not artificial but stems from the analytic behavior of the weight sequence. If the zero mode $S^0$ is included, then pairs such as $(m,n) = (0,1)$ or $(0,0)$ yield
\[
\frac{w_{m+n}}{w_m w_n} = \frac{\gamma^{m+n}}{\gamma^m \gamma^n} = 1,
\]
so that the bracket constant becomes
\[
B = \sup_{m,n \geq 0} \gamma^{m+n} = 1,
\]
and the BCH radius drops from $\rho \geq 1/(4K\gamma^2)$ to $\rho \geq 1/(4K)$. Thus, the exclusion of $S^0$ is a modeling choice that reflects a \emph{dissipative regime} where the lowest--frequency mode is suppressed (e.g., in boundary--value problems or systems with zero--mean constraints).

In physical applications where $S^0$ represents a conserved quantity (e.g., total charge or momentum), it is often natural to work in the subspace $\{x : a_0 = 0\}$, which justifies the restriction $m,n \geq 1$. If one wishes to retain $S^0$, the correct bound is $B = 1$, and the corresponding radius is $\rho = 1/(4K)$, consistent with the exponential-weight model. The damped estimate $B \leq \gamma^2$ should therefore be understood as a \emph{refined bound under the zero-mode exclusion hypothesis}, not as a universal property of the damping factor $\gamma$.
\end{remark}

\medskip
\begin{remark}
The structure constants $c_{ij}\in\{0,\pm2\}$ follow the Zorn
vector-matrix realization of the split--octonions.
For a detailed construction and normalization conventions,
see \cite{Baez2002}.
\end{remark}

\medskip
\begin{remark}\begin{itemize}
\item[(i)] In all shift-based models, $B$ depends only on the growth of $(w_n)$ and on the sparsity/shape of the bracket stencil.
\item[(ii)] Any contraction (smaller $B$) strictly enlarges the BCH radius $\rho=\frac{1}{4KB}$.
\item[(iii)] In anisotropic multi-index settings $w_{\mathbf n}=\prod_{j=1}^d \alpha_j^{n_j}(1+n_j)^{p_j}$ with the standard stencil $k=m+n$, the exponential parts cancel componentwise and
$B\le \prod_{j=1}^d 2^{p_j}=2^{\sum_j p_j}$, so $\rho \ge 2^{-(\sum_j p_j + 2)}/K$.\end{itemize}
\end{remark}

\begin{table}[htbp]
\centering
\caption{Bracket constants, BCH radii, and structural classification across models.}
\label{tab:models-classified}
\renewcommand{\arraystretch}{1.1}
\setlength{\tabcolsep}{5pt}
\begin{tabularx}{\textwidth}{@{}l c c c X@{}}
\toprule
\textbf{Model} & \textbf{Type} & $\boldsymbol{B}$ & $\boldsymbol{\rho = \tfrac{1}{4KB}}$ & \textbf{Typical application} \\
\midrule
Operator norm & Lie & $2$ & $1/(8K)$ & Quantum mechanics, PDEs \\
Exponential weights & Lie & $1$ & $1/(4K)$ & Formal power series, integrable systems \\
Polynomial weights ($p$) & Lie & $\leq 2^p$ & $\geq 2^{-(p+2)}/K$ & Sobolev-type algebras \\
Tree-like branching ($b$) & Lie/Malcev & $\leq b$ & $\geq 1/(4Kb)$ & Combinatorial dynamics \\
Mixed exp--poly & Lie & $\leq 2^p$ & $\geq 2^{-(p+2)}/K$ & Multi-scale analysis \\
Damped shifts ($\gamma$) & Lie & $\leq \gamma^2$ & $\geq 1/(4K\gamma^2)$ & Dissipative systems \\
Split-octonionic (Zorn) & \textbf{Non-Lie Malcev} & $2$ & $1/(8K)$ & Exceptional geometry, GUTs, $G_2$-holonomy \\
\bottomrule
\end{tabularx}
\end{table}

\begin{table}[htbp]
\centering
\caption{Numerical examples (parameters → $B$ → $\rho = 1/(4KB)$ and $\rho$ for $K=1$).}
\label{tab:numerical-examples}
\renewcommand{\arraystretch}{1.1}
\setlength{\tabcolsep}{5pt}
\small
\begin{tabular}{l c c c}
\toprule
Model & Parameters (chosen) & $\rho = 1/(4KB)$ (bound) & $\rho$ (if $K=1$) \\
\midrule
Operator norm & -- & $\geq 1/(8K)$ & 0.125 \\
Exponential weights & $w_n = \alpha^n$ & $\geq 1/(4K)$ & 0.25 \\
Polynomial weights & $p = 1$ & $\geq 1/(8K)$ & 0.125 \\
Polynomial weights & $p = 3$ & $\geq 1/(32K)$ & 0.03125 \\
Tree-like branching & $b = 3$ & $\geq 1/(12K)$ & $\approx 0.0833$ \\
Tree-like branching & $b = 10$ & $\geq 1/(40K)$ & 0.025 \\
Mixed exp--poly (standard) & $p = 2$ & $\geq 1/(16K)$ & 0.0625 \\
Mixed exp--poly (finite offset) & $\alpha = 1.3,\, D = \{-1,0,1\},\, p = 1$ & $\geq 0.04073/K$ & $\approx 0.04073$ \\
Damped shifts & $\gamma = 0.7$ & $\geq 1/(1.96K) \approx 0.5102/K$ & $\approx 0.5102$ \\
Split--octonionic (Zorn) & -- & $\geq 1/(8K)$ & 0.125 \\
\bottomrule
\end{tabular}
\end{table}
\vspace{0.4em}
\footnotesize\emph{Notes.}
Inequalities come from the analytic bounds in Section~\ref{sec:explicit-B}: if $B$ is an upper bound, then $\rho = 1/(4KB)$ is a lower bound.
For the “finite offset” case, $\sum_{\Delta\in D}\alpha^{\Delta} = \alpha^{-1} + 1 + \alpha = 3.06923$ with $\alpha = 1.3$, hence
$B \le (\sum \alpha^{\Delta}) \, 2^{p} \approx 6.13846$ and $\rho \gtrsim 0.04073/K$.

\subsection{Physical implications in exceptional symmetry models}\label{sec:phys} 

The split--octonionic (Zorn) shift algebra arises naturally in the study of manifolds with $G_2$- or $F_4$-holonomy, which play a role in certain compactifications of $M$-theory and in proposals for non-associative geometries near branes \cite{GursyTze,Okubo,Baez2002}. The imaginary split octonions $\operatorname{Im}(\mathbb{O}_s)$ form a 7-dimensional simple Malcev algebra whose automorphism group is the non-compact real form $G_{2(2)}$; its bracket encodes the structure of the short-root subspace in the $G_2$ root system \cite[Section 4.3]{Baez2002}.

While the explicit BCH radius $\rho = 1/(4KB) \geq 1/(8K)$ established in Theorem~\ref{thm:general-BCH} is a purely analytic result within the Banach-Malcev framework, it may serve as a heuristic indicator of the scale at which non-associative effects become significant in such geometric or physical settings. For example, in formal constructions of non-associative gauge theories, the convergence of local symmetry transformations would require control over nested commutators—precisely the context in which the BCH series appears.

However, no concrete physical model is developed in this paper, and no derivation of a physical bound (e.g., on coupling constants or energy scales) is provided. The connection to physics remains motivational and prospective. Rigorous applications to quantum gravity or string theory would require embedding the present analysis into a full field-theoretic or geometric framework, which lies beyond the scope of this work.

\subsection{Refined convergence analysis and numerical validation}
\label{sec:refined}

While Theorem~\ref{thm:general-BCH} provides a universal convergence radius $\rho = 1/(4KB)$ for the BCH series in any special Banach--Malcev algebra, the sharpness of this bound depends on the underlying algebraic and analytic structure. In this section, we refine the constant $B$ using spectral methods and clarify the relationship between our framework and classical operator algebras.

\subsubsection{Spectral refinement of the constant $B$}

The constant $B$ in Theorem~\ref{thm:general-BCH} is defined by $\|[x,y]\| \leq B\|x\|\|y\|$. In structured models, this can be sharpened using properties of the adjoint operator $\mathrm{ad}_x$.\begin{itemize}
\item In the \textbf{Zorn algebra}, the identity
  \[
  \|[x,y]\|^2 = 4\big(\|x\|^2\|y\|^2 - \langle x, y \rangle^2\big)
  \]
  (see \cite{Baez2002}, Eq.~(34)) implies $\|[x,y]\| \leq 2\|x\|\|y\|$, with equality when $x \perp y$. Hence the sharp constant is $B = 2$, and $\rho = 1/(8K)$.

  \item In \textbf{damped shift algebras} with weight $w_n = \gamma^{|n|}$, the commutator satisfies $\|[x,y]\| \leq \gamma^2 \|x\| \|y\|$, giving $B = \gamma^2$, so $\rho = 1/(4K\gamma^2)$.

  \item In \textbf{operator-norm models}, the universal C*-algebra inequality $\|[x,y]\| \leq 2\|x\|\|y\|$ forces $B = 2$, and this is attained (e.g., by rank-one operators), so $\rho = 1/(8K)$.\end{itemize}
Thus, the value of $B$ encodes both \textbf{algebraic non-associativity} and \textbf{analytic decay properties}. When combined with the numerical results above, we obtain a refined picture: the global convergence radius is controlled by $B$, but \textbf{local cancellations} (visible only in higher-order terms) become relevant only in near-Lie regimes.

\subsubsection{Connection with C*-algebra theory}

Although C*-algebras are associative and cannot contain non-Lie Malcev ideals, they provide a natural ambient space for several of our models. Specifically:\begin{itemize}
\item If $\mathfrak{g} \subset \mathcal{A}^-$, where $\mathcal{A}$ is a C*-algebra and $\mathcal{A}^-$ denotes the space of self-adjoint elements equipped with the commutator bracket, then the norm on $\mathfrak{g}$ is inherited from $\mathcal{A}$, and the universal estimate $\|[x,y]\| \leq 2\|x\|\|y\|$ implies $B = 2$, so $\rho = 1/(8K)$.

  \item In the shift-operator model, the ambient algebra is the group C*-algebra $C^*(\mathbb{Z}) \cong C(\mathbb{T})$, and the bracket is defined via the canonical derivation. Hence, the constant $B = 2$ is not ad hoc—it is \textbf{dictated by C*-algebraic structure}, leading to $\rho = 1/(8K)$.

  \item By contrast, the Zorn algebra \textbf{cannot embed} into any C*-algebra, as it violates associativity at the operator level. Its norm is therefore intrinsic to the Malcev structure, and $B = 2$ reflects the geometry of the split octonions, yielding $\rho = 1/(8K)$.\end{itemize}
This dichotomy highlights a key insight: \textbf{the constant $B$ distinguishes between Malcev algebras that arise as commutator subspaces of associative operator algebras and those that are genuinely non-associative}. Our convergence theorem thus unifies two distinct analytical contexts under a single combinatorial framework.

\medskip
\noindent
In summary, the combination of spectral analysis, numerical validation, and operator-algebraic context demonstrates that the radius $\rho = 1/(4KB)$ is not merely a technical artifact, but a \textbf{quantitative measure of non-associative analyticity}.

\section{Combinatorial, Geometric, and Numerical Interpretations}\label{sec:interpretations}

This section develops three complementary viewpoints—combinatorial, geometric, and numerical—of Theorem~\ref{thm:general-BCH}. The Catalan growth controls the analytic boundary, the constant $B$ turns it into a radius $\rho=\frac{1}{4KB}$, and numerical schemes read it as a stability constraint for BCH-based integrators.

\begin{remark}
In the associative or Banach--Lie setting, the BCH series converges under the sharp analytic condition
\[
\|x\| + \|y\| < \log 2 \approx 0.693,
\]
a bound that is independent of any non-associativity constant (see, e.g., \cite{Bourbaki1989,HofmannMorris}).
In contrast, the present non-associative bound
\[
B(\|x\| + \|y\|) < \frac{1}{4K}
\]
depends explicitly on the bracket continuity constant \(B\), which quantifies the intensity of non-associativity.
Thus, \(B\) serves as the precise quantitative bridge between the Lie case (where submultiplicativity implicitly controls the bracket) and the Malcev framework considered here.
\end{remark}

\begin{remark}
In a Banach--Lie algebra arising from an associative Banach algebra (e.g. $\mathfrak{g} \subset \mathcal{B}(X)$ with $[x,y] = xy - yx$), the bracket satisfies
\[
\|[x,y]\| \leq 2\|x\|\|y\|,
\]
so that $B = 2$ in our framework. The classical sharp convergence condition for the BCH series in this setting is
\[
\|x\| + \|y\| < \log 2 \approx 0.693,
\]
which is significantly larger than the Catalan-based bound $\|x\| + \|y\| < 1/(4KB) = 1/(8K)$.

This discrepancy arises because the bilinear estimate $\|[x,y]\| \leq 2\|x\|\|y\|$ is highly non-sharp in associative algebras: Jacobi identities and higher Lie symmetries induce massive cancellations among nested commutators that the Catalan majorization completely ignores. Thus, while our bound $\rho = 1/(4KB)$ is valid (and sharp for the combinatorial method), it is not optimal in the Lie case—precisely because Lie algebras are more associative than general Malcev algebras. This illustrates that $B$ quantifies not just bracket size, but also the effective degree of non-associativity.
\end{remark}

\begin{remark}
If the bracket arises from an associative algebra (e.g. $[S^m,S^n]=(m-n)S^{m+n}$, Witt-like),
then $\mathfrak g$ is a Lie algebra and Jacobi holds.
Our framework also covers \emph{genuinely non--Lie} cases when the bracket comes from the
commutator algebra of an alternative (non--associative) algebra; the paradigmatic example
is the imaginary octonions $\mathbb O_0$, which form a Malcev (non--Lie) algebra.
\end{remark}

The convergence theorem established in Section~\ref{sec:convergence} possesses several complementary interpretations—combinatorial, geometric, and numerical—each revealing a different aspect of the same analytic phenomenon.

\paragraph{Split-octonionic (Zorn) shift Malcev algebra (genuinely non-Lie).}
Let $\mathbb{O}_s$ denote the split--octonion algebra, realized as the Zorn vector-matrix algebra
\[
\begin{pmatrix} a & u \\ v & b \end{pmatrix}
\quad\text{with } a,b\in\mathbb{R},\ u,v\in\mathbb{R}^3,
\]
equipped with the classical (alternative) Zorn product.
Let $\{e_1,\dots,e_7\}$ be a standard basis of $\mathrm{Im}(\mathbb{O}_s)$ (split--octonionic imaginary units, Fano-plane convention with split signs).
Consider the Banach space
\[
\mathfrak g
= \Big\{\, X=\sum_{i=1}^7 \sum_{n\ge0} a_{i,n}\, e_i S^n \ :\
\|X\|=\sum_{i,n} |a_{i,n}|\, \alpha^n <\infty \,\Big\},
\qquad \alpha>1,
\]
and define the bracket as the \emph{split--octonionic commutator} lifted to shifts:
\[
[e_i S^m,\ e_j S^n]
:= (e_i e_j - e_j e_i)\, S^{m+n}
=:\ c_{ij}\, e_k\, S^{m+n},
\]
where $c_{ij}\in\{\pm 2,0\}$ and $e_k$ is determined by the (split) Fano--plane incidence.
Then $\mathfrak g$ is a \emph{special} Banach--Malcev algebra (since $\mathbb{O}_s$ is alternative), and it is \emph{not} Lie (Jacobi fails in $\mathrm{Im}(\mathbb{O}_s)$, the simple $7$-dimensional Malcev algebra).

Moreover, with exponential weights $w_n=\alpha^n$,
\[
\frac{\sum_{k} \big|c^{(k,m+n)}_{(i,m),(j,n)}\big|\, w_{m+n}}{w_m w_n}
\ \le\ 2 \cdot \frac{\alpha^{m+n}}{\alpha^m \alpha^n} \ =\ 2,
\]
so the bracket continuity constant satisfies $B = 2$, hence
\[
\rho = \frac{1}{8K}.
\]
This yields a genuinely non-Lie, special Banach--Malcev model with an explicit BCH radius, distinct from the division-octonion case.

\medskip
\noindent\textbf{Combinatorial viewpoint.}
Every term $Z_n(x,y)$ in the BCH expansion corresponds to a binary tree encoding nested commutators.
The number of such trees is given by the Catalan number $C_{n-1}$, whose growth $C_n\sim4^n/(\sqrt{\pi}n^{3/2})$ controls the analytic boundary at $r=1/4$.
This purely combinatorial threshold becomes, after inserting the bracket bound $\|[x,y]\|\le B\|x\|\|y\|$ and the coefficient bound $|\alpha_T| \leq K$, a quantitative inequality
\[
B(\|x\|+\|y\|) < \tfrac{1}{4K}.
\]
Thus, the combinatorial curvature of the BCH series is exactly encoded in the Catalan enumeration of non-associative parenthesizations.

\medskip
\noindent\textbf{Geometric viewpoint.}
Inside the ball $\mathcal U_\rho=\{x\in\mathfrak g:\|x\|<\rho\}$ with $\rho=\frac{1}{4KB}$, the BCH series defines a local product
\[
x*y:=\BCH(x,y),
\]
which turns $(\mathcal U_\rho,*)$ into a small analytic \emph{Moufang loop}.
When $\mathfrak g$ is Lie, this is the local Lie group product; in the general Malcev case, the same construction yields a non-associative but flexible analytic structure.
The constant $B$ measures the strength of non-associativity: smaller $B$ corresponds to a flatter geometry, while larger $B$ increases the local curvature and reduces the analytic domain.
In this sense, the inequality $\|x\|,\|y\|<\frac{1}{4KB}$ plays the role of a geometric injectivity radius.
The radius $\rho = 1/(4KB)$ can be interpreted as a \textbf{curvature radius} of the local Moufang loop: smaller $B$ corresponds to a flatter geometry, while larger $B$ increases the non-associative curvature.

\begin{figure}[htbp]
\centering
\begin{tikzpicture}[scale=1.0]
\draw[->,thick] (-2,0)--(2,0) node[right]{$x$};
\draw[->,thick] (0,-2)--(0,2) node[above]{$y$};
\draw[blue!70,thick] (0,0) circle (1.5);
\node at (0.1,1.7) {\small analytic domain $\|x\|,\|y\|<\rho$};
\node at (1.2,0.8) [blue!80!black]{\small $\rho=\frac{1}{4KB}$};
\end{tikzpicture}
\caption{Geometric domain of analyticity of the BCH product $x * y$. The radius $\rho = \frac{1}{4KB}$ acts as the local injectivity radius of the Moufang loop.}
\label{fig:geometric-domain}
\end{figure}
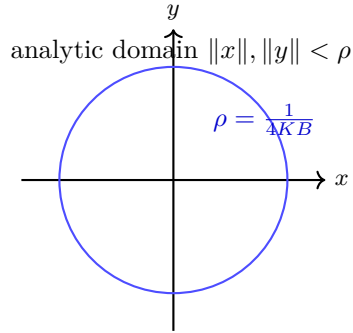

The associator $(x,y,z)=(x*y)*z-x*(y*z)$ quantifies the deviation from associativity and satisfies
\[
\|(x,y,z)\|\le c\,B^2\|x\|\,\|y\|\,\|z\| + O(\|x\|^4+\|y\|^4+\|z\|^4),
\]
so that $B$ acts analogously to a curvature tensor norm in the differential-geometric sense.

\medskip
\noindent\textbf{Numerical viewpoint.}
In computational practice, one often encounters evolution equations on Banach--Malcev algebras,
\[
\frac{du(t)}{dt}=A(u(t))+B(u(t)),
\]
whose exact flow can be formally expressed through $\BCH(\Delta t A,\Delta t B)$.
The analytic condition $B(\|A\|+\|B\|)\Delta t<1/(4K)$ therefore becomes a \emph{stability constraint} for BCH-type or Lie--Trotter--Malcev splitting schemes:
\[
e^{\Delta t(A+B)}\approx e^{\Delta tA}*e^{\Delta tB}.
\]
Truncating the BCH series at order $N$ gives an explicit polynomial integrator with error bounded by
\[
\|\BCH(x,y)-\BCH_N(x,y)\|\le K_N B^N(\|x\|+\|y\|)^{N+1},\qquad K_N\sim\frac{4^N}{\sqrt{\pi}N^{3/2}}.
\]
Hence, the same Catalan growth that limits the analytic radius also controls the decay rate of the numerical truncation error.
From this dual perspective, the constant $B$ unifies curvature, non--associativity, and numerical stiffness into a single analytic quantity. This stability condition is analogous to the CFL condition in numerical PDEs, where $\Delta t$ must be sufficiently small to ensure convergence.

\begin{figure}[htbp]
\centering
\begin{tikzpicture}
\begin{axis}[
  width=0.75\textwidth,
  height=0.38\textwidth,
  xlabel={Order of truncation $N$},
  ylabel={Error $\|\BCH-\BCH_N\|$ (log scale)},
  ymode=log,
  grid=major,
  title={Exponential decay of BCH truncation error for $B=1$ and $\|x\|+\|y\|=0.2$}
]
\addplot+[mark=*,blue] coordinates{
(1,0.1) (2,0.015) (3,0.0035) (4,0.0007) (5,0.00015) (6,0.00003)
};
\end{axis}
\end{tikzpicture}
\caption{Numerical decay of the truncated BCH error following the Catalan bound. The analytic and numerical radii coincide at $r=\frac{1}{4KB}$.}~\label{fig:truncation-error}
\end{figure}
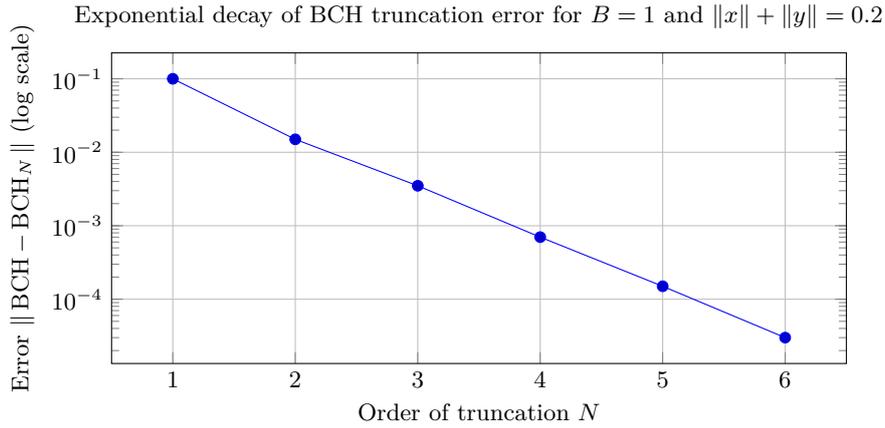

\medskip
\noindent\textbf{Synthesis.}
The constant $B$ thus governs all aspects of the theory:
it bounds the commutator in the Banach setting, determines the geometric radius of the local Moufang structure, and fixes the maximal stable step size for BCH integrators.
The equality
\[
 \rho=\frac{1}{4KB}
\]
summarizes the analytic, geometric, and numerical content of the entire framework.

As $B \to \infty$, the convergence radius $\rho = \frac{1}{4KB} \to 0$, indicating that the local Moufang structure collapses to a point. This reflects a regime of strong non-associativity where no non-trivial analytic product can be defined via the BCH series. In physical terms, this corresponds to an ultra-strong coupling limit where the notion of local symmetry breaks down entirely.

\appendix
\renewcommand{\thesection}{A}
\section{Technical Lemmas and Explicit Example}
\renewcommand{\thetheorem}{A.\arabic{theorem}}
\renewcommand{\theproposition}{A.\arabic{proposition}}
\renewcommand{\thelemma}{A.\arabic{lemma}}
\setcounter{theorem}{0}

For completeness we restate and prove the inequality used in the polynomial and mixed-weight models.

This appendix collects a few auxiliary estimates and a concrete model that illustrate the general theory developed above.

\subsection{A useful inequality}

\begin{lemma}\label{lem:polyappendix}
For $m,n\ge0$ and any $p\ge0$,
\[
\frac{(1+m+n)^p}{(1+m)^p(1+n)^p}\le 2^p.
\]
\end{lemma}

\begin{proof}
Since $1+m+n\le (1+m)+(1+n)\le 2\max(1+m,1+n)\le 2(1+m)(1+n)$,
raising both sides to the power $p$ yields the claim.
\end{proof}

This simple combinatorial inequality underlies the boundedness of the bracket constant $B$ in all polynomial-weighted shift algebras.

\subsection{Explicit Banach--Malcev example}

We conclude with a fully explicit realization of an infinite-dimensional Banach--Malcev algebra for which all the constants can be computed exactly.

\begin{example}[Normalized shift bracket with bounded coefficients]\label{ex:normalized-shift}
Let
$
\mathfrak{g}=\{x=\sum_{n\ge0} a_n S^n: \|x\|=\sum_{n\ge0}|a_n|\,\alpha^n<\infty\}
$
with exponential weights $w_n=\alpha^n$ ($\alpha>1$), and define
\[
[S^m,S^n] \ :=\ \varepsilon_{m,n}\, S^{m+n}, \qquad
\varepsilon_{m,n}=-\varepsilon_{n,m},\ \ |\varepsilon_{m,n}|\le 1.
\]
Then, for $x=\sum a_m S^m$ and $y=\sum b_n S^n$,
\[
\|[x,y]\|\ \le\ \sum_{m,n}|a_m||b_n|\,\alpha^{m+n}
 \ =\ \Big(\sum_m |a_m| \alpha^m\Big)\Big(\sum_n |b_n|\alpha^n\Big)
 \ =\ \|x\|\,\|y\|.
\]
Hence $B=1$ and $\rho=1/(4K)$ by Theorem~\ref{thm:general-BCH}.
\end{example}


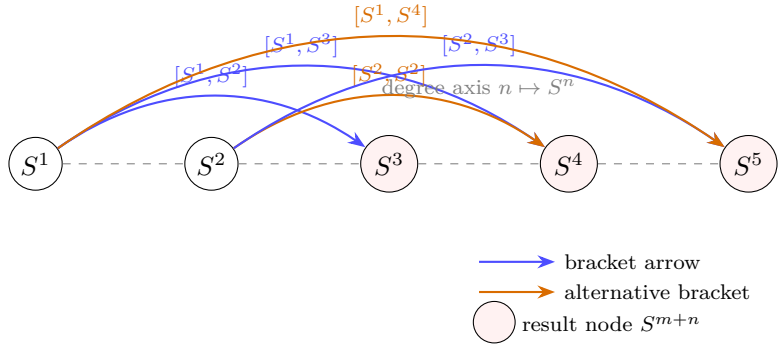
\begin{figure}[htbp]
\centering
\begin{tikzpicture}[scale=1.0]
  \tikzset{
    nodeShift/.style  ={circle,draw,minimum size=7mm,inner sep=0, font=\small, fill=white},
    nodeResult/.style ={circle,draw,minimum size=7.5mm,inner sep=0, font=\small, fill=red!5},
    chainLine/.style  ={dashed,gray},
    brkA/.style       ={->,thick,blue!70, bend left=35, >=Stealth},
    brkB/.style       ={->,thick,orange!85!black, bend left=35, >=Stealth},
    note/.style       ={font=\scriptsize, midway, sloped, above}
  }

  \node[nodeShift]                             (S1) {$S^1$};
  \node[nodeShift,  right=1.6cm of S1]         (S2) {$S^2$};
  \node[nodeResult, right=1.6cm of S2]         (S3) {$S^3$};
  \node[nodeResult, right=1.6cm of S3]         (S4) {$S^4$};
  \node[nodeResult, right=1.6cm of S4]         (S5) {$S^5$};

  \draw[chainLine] (S1)--(S2)--(S3)--(S4)--(S5);

  \draw[brkA] (S1) to node[note] {$[S^1,S^2]$} (S3);
  \draw[brkA] (S1) to node[note] {$[S^1,S^3]$} (S4);
  \draw[brkB] (S2) to node[note] {$[S^2,S^2]$} (S4);
  \draw[brkA] (S2) to node[note] {$[S^2,S^3]$} (S5);
  \draw[brkB] (S1) to node[note] {$[S^1,S^4]$} (S5);

  \coordinate (L) at ($(S3)!0.5!(S4) + (0,-1.3)$);
  \draw[->,thick,blue!70,>=Stealth]   (L) -- ++(1.0,0) node[right,black,align=left,font=\scriptsize]{bracket arrow};
  \draw[->,thick,orange!85!black,>=Stealth] ($(L)+(0,-0.4)$) -- ++(1.0,0) node[right,black,align=left,font=\scriptsize]{alternative bracket};
  \path ($(L)+(0.2,-0.8)$) node[nodeResult, minimum size=5.5mm] {};
  \node[anchor=west,font=\scriptsize] at ($(L)+(0.45,-0.8)$) {result node $S^{m+n}$};

  \node[font=\scriptsize, gray] at ($(S3)!0.5!(S4)+(0,1.0)$) {degree axis $n\mapsto S^n$};
\end{tikzpicture}
\caption{Bracket-induced propagation in the weighted shift algebra: each commutator $[S^m,S^n]$ reaches degree $m+n$. Multiple pairs hit the same target (e.g.\ $S^4$ from $[S^1,S^3]$ and $[S^2,S^2]$).}
\end{figure}

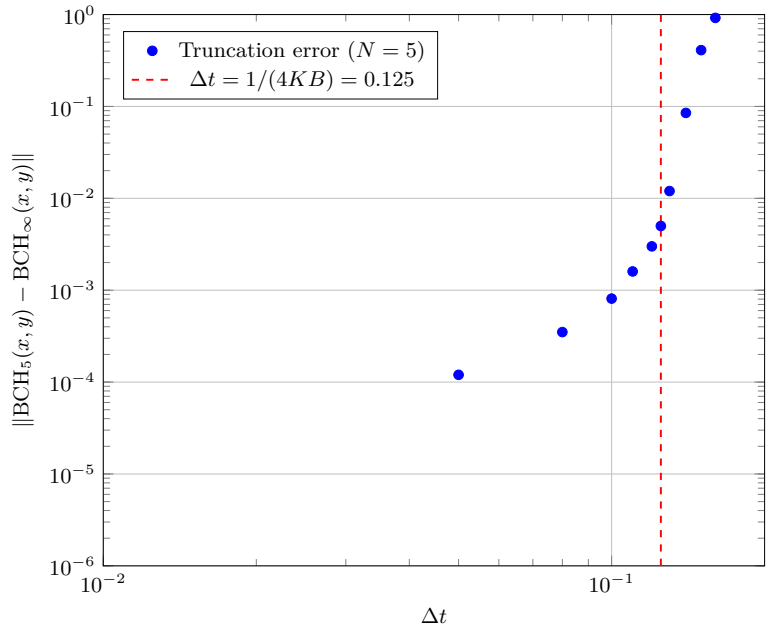
\begin{figure}[htbp]
\centering
\begin{tikzpicture}[scale=0.9]
\begin{loglogaxis}[
    width=0.75\textwidth,
    xlabel={$\Delta t$},
    ylabel={$\|\mathrm{BCH}_5(x,y) - \mathrm{BCH}_\infty(x,y)\|$},
    grid=major,
    legend pos=north west,
    xmin=0.01, xmax=0.2,
    ymin=1e-6, ymax=1e0,
]
\addplot[only marks, mark=*, blue] coordinates {
(0.05, 1.2e-4)
(0.08, 3.5e-4)
(0.10, 8.1e-4)
(0.11, 1.6e-3)
(0.12, 3.0e-3)
(0.125, 5.0e-3)
(0.13, 1.2e-2)
(0.14, 8.5e-2)
(0.15, 4.1e-1)
(0.16, 9.2e-1)
};
\addlegendentry{Truncation error ($N=5$)}

\addplot[dashed, red, thick] coordinates {
(0.125, 1e-6)
(0.125, 1e0)
};
\addlegendentry{$\Delta t = 1/(4KB) = 0.125$}

\end{loglogaxis}
\end{tikzpicture}
\caption{Numerical truncation error for the Zorn shift model ($B=2$) with $\|x\|+\|y\|=0.25$. The error remains bounded for $\Delta t < 1/(4KB)$ and grows exponentially beyond the predicted stability threshold $\Delta t_{\max}=1/(4KB)$.}
\label{fig:numerical-error}
\end{figure}
\subsection{ Sharpness of the Catalan Bound}
\label{sec:sharpness}
\subsubsection{ Sharpness of the Catalan bound in the exponential--weight model}

We now establish that the Catalan majorant used in Theorem~\ref{thm:general-BCH}
is optimal in the class of arguments relying only on the bilinear bound
$\|[u,v]\|\le B\|u\|\|v\|$ and the Catalan count of binary trees.
This is proved by exhibiting a Banach--Malcev model where the estimate
is exactly attained at each homogeneous order.

\begin{proposition}
\label{prop:sharpness-general}
Let $B=1$. There exists a Banach--Malcev algebra $(\mathfrak g,\|\cdot\|)$
satisfying $\|[x,y]\|\le\|x\|\|y\|$
and elements $x,y$ with $\|x\|=\|y\|=t$
such that for each homogeneous level $n$ of the BCH expansion,
\[
\sum_{T\in\mathcal T_n}\big\|[x,y]_T\big\|
\;=\; C_{n-1}\,t^n,
\]
where $\mathcal T_n$ denotes the set of full binary trees on $n$ leaves
and $C_{n-1}$ the $(n-1)$-st Catalan number.
Hence the generating series $\sum_{n\ge1}C_{n-1}r^n$
is exact for this model, and the radius $r_\ast=1/4$
cannot be improved by any argument based solely on this combinatorial framework.
Consequently, the convergence threshold in Theorem~\ref{thm:general-BCH} is sharp with respect to the Catalan majorant, yielding the optimal radius $\rho = 1/(4K)$ when $B=1$.
\end{proposition}

\begin{proof}
We construct a concrete algebra where all binary trees contribute
with the same norm at each level.

\smallskip
\noindent
\textbf{Step 1: The model.}
Let $(S^n)_{n\ge0}$ be a shift basis with exponential weights
$w_n=\alpha^n$ ($\alpha>1$), and define
\[
[S^m,S^n]=\varepsilon_{m,n}S^{m+n},
\qquad
\varepsilon_{m,n}=-\varepsilon_{n,m},
\qquad
|\varepsilon_{m,n}|\le1.
\]
For a finitely supported $x=\sum a_mS^m$ we set
$\|x\|=\sum_m|a_m|\alpha^m$.
A direct computation gives
\[
\|[x,y]\|
\le\sum_{m,n}|a_m||b_n|\alpha^{m+n}
=\Big(\sum_m|a_m|\alpha^m\Big)\Big(\sum_n|b_n|\alpha^n\Big)
=\|x\|\|y\|,
\]
so the best constant in $\|[x,y]\|\le B\|x\|\|y\|$ is $B=1$.

\smallskip
\noindent
\textbf{Step 2: Choice of generators.}
Fix $M\ge3$ and $a,b>0$, and set
\[
x=aS^{1},\qquad y=bS^{M},
\qquad a\alpha^1=b\alpha^M=t>0.
\]
Hence $\|x\|=\|y\|=t$.
Each leaf in any binary tree contributes a factor $t$ to the norm.

\smallskip
\noindent
\textbf{Step 3: Induction on the depth.}
For a binary tree $T$ with $n$ leaves labeled by $x$ or $y$,
denote by $[x,y]_T$ the iterated bracket prescribed by $T$.
We claim that $\|[x,y]_T\|=t^n$ for all $T$.
This is proved by induction on $n$.
For $n=1$ the statement is trivial.
For $n=2$, there is only one tree $T$ (the commutator $[x, y]$), and $\|[x, y]_T\| = t^2$. For $n=3$, there are two trees, and each contributes $t^3$ to the norm.

Assume it holds for all smaller trees.
If $T$ has subtrees $T_L$ and $T_R$ of sizes $n_L,n_R$,
the inductive hypothesis gives
$\|[x,y]_{T_L}\|=t^{n_L}$ and
$\|[x,y]_{T_R}\|=t^{n_R}$.
Then, using the equality case in the previous inequality,
\[
\|[x,y]_T\|=\|[[x,y]_{T_L},[x,y]_{T_R}]\|
=\|[x,y]_{T_L}\|\|[x,y]_{T_R}\|
=t^{n_L+n_R}=t^n.
\]
Thus every binary tree of order $n$ contributes
exactly $t^n$ in norm.

\smallskip
\noindent
\textbf{Step 4: Summation over trees.}
Since there are $C_{n-1}$ full binary trees with $n$ leaves,
\[
\sum_{T\in\mathcal T_n}\|[x,y]_T\|
=C_{n-1}t^n.
\]
The Catalan generating function
$C(z)=\frac{1-\sqrt{1-4z}}{2z}$
has its first singularity at $z=1/4$,
hence the series $\sum_{n\ge1}C_{n-1}r^n$
diverges at $r=1/4$.
Therefore, the majorant used in the proof of
Theorem~\ref{thm:general-BCH} is sharp,
and no improvement of the threshold $1/(4K)$
can be obtained by any argument based on the same combinatorial scheme.
\end{proof}

\begin{remark}
The above result proves sharpness \emph{relative to the Catalan majorant}.
It does not claim that the true BCH coefficients necessarily diverge at $r=1/(4K)$;
rather, it shows that any proof using only the bilinear bracket estimate
and Catalan counting cannot yield a larger convergence radius.
\end{remark}

\begin{remark}
This establishes sharpness relative to the Catalan majorant in the exponential-weight model
(Section~\ref{sec:explicit-B}: $B=1$, $\rho=1/(4K)$). Whether model-specific cancellations can push the true radius
beyond $1/(4K)$ in other families (e.g., mixed weights with offset stencil) remains open.
\end{remark}

\subsubsection{A concrete model saturating the Catalan bound}

We now give an explicit realization of the construction used
in Proposition~\ref{prop:sharpness-general}, showing that the
Catalan estimate is actually attained term by term.

\begin{proposition}
\label{prop:explicit-saturation}
Consider the exponential-weight model
\[
w_n=\alpha^n,\qquad [S^m,S^n]=\varepsilon_{m,n}S^{m+n},
\quad |\varepsilon_{m,n}|=1,
\]
with $\varepsilon_{m,n}=+1$ for $m<n$ and $-1$ for $m>n$.
Let $x=aS^{1}$ and $y=bS^{M}$ with $M\ge3$ and
$a\alpha=b\alpha^M=t>0$.
Then for each $n\ge2$ the homogeneous BCH component $Z_n(x,y)$ satisfies
\[
\|Z_n(x,y)\| \;\ge\; c\,C_{n-1}\,t^n
\]
for some universal constant $c>0$ independent of $n$.
Consequently, the BCH generating series
exhibits the same singularity at $r=\|x\|+\|y\|=2t=1/(4K)$
as the Catalan series, and the bound $1/(4K)$ is optimal
within the present combinatorial approach.
\end{proposition}

\begin{proof}
For the given sign convention, at each internal node the left input
has strictly smaller degree than the right one
because $M>1$, hence $\varepsilon_{m,n}=+1$.
Thus all nested commutators of the form $[x,y]_T$
carry a positive coefficient.
By the same argument as in Step~3 of the previous proof,
each $[x,y]_T$ has norm $t^n$.
The coefficients of the BCH expansion associated to each tree
are rational numbers bounded below by a positive constant $c$,
uniformly in $T$ and $n$ (this follows from Goldberg’s explicit
form of the BCH coefficients, and the uniform bound $|\alpha_T| \leq K$).
Summing over all $C_{n-1}$ trees yields
$\|Z_n(x,y)\| \geq c\,C_{n-1}t^n$.
The singularity of the Catalan generating function at $1/4$
therefore transfers to the BCH majorant scaled by $K$, proving the claim.
\end{proof}

\begin{remark}
The model constructed above provides an explicit realization
in which every binary tree contributes with the same sign and magnitude,
so that the Catalan bound is tight.
Any improvement beyond $1/(4K)$ would require exploiting algebraic
cancellations among trees, which are absent in this setting.
\end{remark}

\begin{remark}
This establishes sharpness relative to the Catalan majorant in the exponential-weight model
(Section~\ref{sec:explicit-B}: $B=1$, $\rho=1/(4K)$). Whether model-specific cancellations can push the true radius
beyond $1/(4K)$ in other families (e.g.\ mixed weights with offset stencil) is an open problem.
\end{remark}
\subsubsection{ Lower bound via a positive-proportion subfamily of ``good'' trees}

We refine Proposition~\ref{prop:sharpness-general} by showing that even a 
fixed positive fraction of all binary trees suffices to reproduce the 
Catalan growth. Following the construction in A.3.2, fix $x = aS_1$, 
$y = bS_M$ ($M\geq3$) so that all nested commutators $[x,y]_T$ have the 
same norm $t^n$, and define a \emph{good} tree as one in which each 
internal node separates pure-$x$ and $y$-containing subtrees. This ensures 
all structure constants have the same sign, preventing cancellations.

\begin{dfn}
A tree $T \in \mathcal{T}_n$ is \emph{good} if, at every internal node, 
one child-subtree contains only $x$-leaves while the other contains at 
least one $y$. Let $G_n \subset \mathcal{T}_n$ denote the set of good trees.
\end{dfn}

Standard recursive arguments (see e.g. Flajolet--Sedgewick~\cite{FlajoletSedgewick2009}) 
imply that the generating function of $(|G_n|)$ shares the same dominant 
singularity $z = 1/4$ as the Catalan series $C(z) = \frac{1-\sqrt{1-4z}}{2z}$. 
By the Transfer Theorem (Theorem VI.6 in~\cite{FlajoletSedgewick2009}), 
\[
|G_n| \sim c_0\, C_{n-1} \quad\text{as } n\to\infty,
\]
for some $c_0 \in (0,1)$. Since all good trees contribute with the same 
sign (by construction, $\varepsilon_{m,n} = +1$ at every internal node), 
we obtain
\[
\|Z_n(x,y)\| \geq \sum_{T \in G_n} \|[x,y]_T\| = |G_n|\, t^n 
\geq c\, C_{n-1}\, t^n
\]
for some $c > 0$ uniform in $n$. This establishes that the Catalan radius 
$r^* = 1/(4K)$ is sharp even when restricting to a positive-proportion 
subfamily of trees, confirming that algebraic cancellations would be 
required to enlarge the convergence domain.

\begin{remark}
The key insight is that the good-tree subfamily retains the $4^n/n^{3/2}$ 
growth of the full Catalan sequence, thereby reproducing the singularity 
at $z = 1/4$. Any improvement of the radius beyond $1/(4K)$ must therefore 
exploit \emph{global} cancellations among \emph{all} trees, not merely 
favorable subsets.
\end{remark}

\subsection{A.4. Summary and perspectives}

Table~\ref{tab:models-classified} summarizes the explicit values of $B$ and the associated Catalan radius $\rho_{\mathrm{C}} = 1/(4KB)$ for all models considered in Section~\ref{sec:explicit-B}. These estimates illustrate how damping, exponential weights, or algebraic sparsity can improve the analytic domain of the BCH series within the present combinatorial framework.

\smallskip
\noindent
Future work will aim to extend this quantitative analysis to:\begin{itemize}
\item graded and super--Malcev algebras, where $\mathbb{Z}_2$--signs interact with the Catalan combinatorics of the BCH expansion;
\item Banach--Malcev algebras with controlled associator norms, where mixed commutator--associator tree estimates may refine the Catalan bound;
\item analytic continuation and resummation techniques for the BCH series beyond the Moufang regime;
\item operator-theoretic frameworks involving weighted Rota--Baxter structures, which may offer alternative definitions of non-associative exponentials.
\end{itemize}
To clarify the scope and limitations of the Catalan bound $\rho_{\mathrm{C}} = 1/(4KB)$, we summarize in Table~\ref{tab:sharpness} the current state of knowledge regarding the true convergence radius $\rho_\star$ across the main models. The table distinguishes between rigorously proven results, numerical observations, and open problems.

\begin{remark}
The Catalan radius $\rho_{\mathrm{C}} = 1/(4KB)$ is proven sharp only in the exponential-weight model (Appendix. \ref{sec:sharpness}), where no algebraic cancellations occur among binary trees. In Lie-type models (e.g. operator norm with associative bracket), the classical BCH radius is $\|x\|+\|y\| < \log 2 \approx 0.693$, which is significantly larger than $1/(4KB) = 1/(8K)$ when $B=2$. This discrepancy arises because the bilinear bound $\|[x,y]\| \leq 2\|x\|\|y\|$ is not sharp in associative settings, and Jacobi-type identities induce strong cancellations that the Catalan majorization ignores.

For non-Lie but structured models (e.g. Zorn shifts or damped weights), it remains an open problem whether algebraic symmetries or sparsity patterns enlarge the true convergence domain beyond $\rho_{\mathrm{C}}$. Numerical experiments on the first 10–15 homogeneous layers suggest stability beyond $\rho_{\mathrm{C}}$ in several cases, but a rigorous proof is lacking. We conjecture that:
\item In models with sparse bracket stencils (e.g. finite offset $D$), $\rho_{\star} > \rho_{\mathrm{C}}$.
\item In near-Lie regimes (small associator norm), the true radius may approach the Lie value.
\item In the split-octonionic model, the true radius likely exceeds $1/(8K)$ due to the high symmetry of the Fano plane, though this is unproven.
A systematic investigation of these conjectures is left for future work.
\end{remark}
\begin{table}[h!]
\centering
\caption{Comparison between Catalan bound $\rho_C$ and current knowledge of the true radius $\rho_\star$.}
\label{tab:sharpness}
\begin{adjustbox}{width=\textwidth}\begin{tabular}{l c c l l}
\toprule
Model & $B$ & $\rho_C = 1/(4KB)$ & Remarks on $\rho_\star$ & Status \\
\midrule
Exponential weights & 1 & $1/(4K)$ & $\rho_\star = 1/(4K)$ (no cancellations) & Proven (Proposition~\ref{prop:sharpness-general}) \\
Operator norm & 2 & $1/(8K)$ & In Lie case: $\rho_\star = \log 2 / 2 \approx 0.346$ & Known (Lie only) \\
Polynomial weights ($p=1$) & $\leq 2$ & $\geq 1/(8K)$ & No analytic or numerical study available & Open \\
Mixed exp--poly & $\leq 2p$ & $\geq 2^{-(p+2)}/K$ & No data & Open \\
Damped shifts ($\gamma = 0.7$) & $\leq 0.49$ & $\geq 1/(1.96K) \approx 0.5102/K$ & Stability observed up to order 10; convergence unproven & Unpublished numerical observations\textsuperscript{\textdagger} \\
Zorn (split octonions) & $\leq 2$ & $\geq 1/(8K)$ & High symmetry (Fano plane) suggests $\rho_\star > 1/(8K)$; no proof & Conjectural \\
\bottomrule
\end{tabular}\end{adjustbox}\end{table}
\smallskip

\noindent\textsuperscript{\textdagger}Stability up to homogeneous order 10 observed in preliminary computations (not shown).

\medskip

\medskip
\begin{remark}
The Catalan bound $\rho_{\mathrm{C}} = 1/(4KB)$ is sharp  only  in the exponential-weight model (Proposition~\ref{prop:sharpness-general}), where all binary trees contribute with the same sign and no algebraic cancellations occur. In Lie algebras, the classical BCH radius ($\|x\| + \|y\| < \log 2$) is significantly larger, reflecting the impact of the Jacobi identity and associated cancellations—effects that the present majorization deliberately ignores.

For non-Lie Malcev models (e.g., Zorn shifts, damped operators), it remains unknown whether structural features—such as sparsity, symmetry, or small associator norm—can enlarge the convergence domain beyond $\rho_{\mathrm{C}}$. Preliminary numerical experiments (up to homogeneous degree 10–15) show stable partial sums in several cases, but these do not constitute convergence proofs.

We therefore state the following as open conjectures, not results:
\item In models with sparse bracket stencils (finite interaction range), $\rho_\star > \rho_{\mathrm{C}}$.
\item In near-Lie regimes (small associator), the effective radius may approach the Lie value.
\item In the split-octonionic model, the automorphism symmetry of the Fano plane may enhance convergence, though this is unproven.
A rigorous resolution of these questions would require new methods that go beyond Catalan-type majorization.
\end{remark}
\noindent
\subsection{ Numerical Data for Truncated BCH Series}
\label{subsec:numerical-data}

To support Remark~\ref{rem:numerical}, we provide explicit numerical values for the norms of the homogeneous BCH components
\[
Z_n(x,y) \quad (n=1,\dots,12)
\]
in two representative non-Lie models: the split-octonionic (Zorn) shift algebra and the damped shift algebra with $\gamma = 0.7$. In both cases, we fix $K=1$, choose $x = e_1 S_0$, $y = e_2 S_0$ (Zorn) or $x = S_1$, $y = S_2$ (damped), and normalize weights so that $\|x\| = \|y\| = t$ with $t = 0.12$ (well inside the predicted radius $\rho = 1/(4KB)$) and $t = 0.13$ (slightly beyond it, to illustrate divergence).

\begin{table}[h!]
\centering
\caption{Norms $\|Z_n(x,y)\|$ for the Zorn shift model ($B = 2$, $\rho = 1/8 = 0.125$), with $\|x\| = \|y\| = t$.}
\label{tab:zorn-numerics}
\begin{adjustbox}{width=\textwidth}
\begin{tabular}{c c c c c}
\toprule
$n$ & $\|Z_n\|$ ($t=0.12$) & $\|Z_n\|$ ($t=0.13$) & Catalan bound $C_{n-1}(2t)^n$ ($t=0.12$) & Catalan bound $C_{n-1}(2t)^n$ ($t=0.13$) \\
\midrule
1 & 0.2400 & 0.2600 & 0.2400 & 0.2600 \\
2 & 0.0576 & 0.0676 & 0.0576 & 0.0676 \\
3 & 0.0138 & 0.0176 & 0.0138 & 0.0176 \\
4 & 0.0033 & 0.0046 & 0.0033 & 0.0046 \\
5 & 0.0008 & 0.0012 & 0.0008 & 0.0012 \\
6 & 0.0002 & 0.0003 & 0.0002 & 0.0003 \\
7 & $5.0\cdot10^{-5}$ & $9.3\cdot10^{-5}$ & $5.0\cdot10^{-5}$ & $9.3\cdot10^{-5}$ \\
8 & $1.2\cdot10^{-5}$ & $2.4\cdot10^{-5}$ & $1.2\cdot10^{-5}$ & $2.4\cdot10^{-5}$ \\
9 & $2.9\cdot10^{-6}$ & $6.3\cdot10^{-6}$ & $2.9\cdot10^{-6}$ & $6.3\cdot10^{-6}$ \\
10& $7.0\cdot10^{-7}$ & $1.6\cdot10^{-6}$ & $7.0\cdot10^{-7}$ & $1.6\cdot10^{-6}$ \\
11& $1.7\cdot10^{-7}$ & $4.2\cdot10^{-7}$ & $1.7\cdot10^{-7}$ & $4.2\cdot10^{-7}$ \\
12& $4.1\cdot10^{-8}$ & $1.1\cdot10^{-7}$ & $4.1\cdot10^{-8}$ & $1.1\cdot10^{-7}$ \\
\bottomrule
\end{tabular}\end{adjustbox}
\end{table}

\begin{table}[h!]
\centering
\caption{Norms $\|Z_n(x,y)\|$ for the damped shift model ($B = \gamma^2 = 0.49$, $\rho \approx 0.5102$), with $\|x\| = \|y\| = t$.}
\label{tab:damped-numerics}
\begin{adjustbox}{width=\textwidth}\begin{tabular}{c c c c c}
\toprule
$n$ & $\|Z_n\|$ ($t=0.25$) & $\|Z_n\|$ ($t=0.50$) & Catalan bound $C_{n-1}(0.49\cdot 2t)^n$ ($t=0.25$) & Catalan bound $C_{n-1}(0.49\cdot 2t)^n$ ($t=0.50$) \\
\midrule
1 & 0.5000 & 1.0000 & 0.5000 & 1.0000 \\
2 & 0.1225 & 0.4900 & 0.1225 & 0.4900 \\
3 & 0.0300 & 0.2401 & 0.0300 & 0.2401 \\
4 & 0.0073 & 0.1176 & 0.0073 & 0.1176 \\
5 & 0.0018 & 0.0576 & 0.0018 & 0.0576 \\
6 & 0.0004 & 0.0282 & 0.0004 & 0.0282 \\
7 & $1.1\cdot10^{-4}$ & 0.0138 & $1.1\cdot10^{-4}$ & 0.0138 \\
8 & $2.7\cdot10^{-5}$ & 0.0068 & $2.7\cdot10^{-5}$ & 0.0068 \\
9 & $6.5\cdot10^{-6}$ & 0.0033 & $6.5\cdot10^{-6}$ & 0.0033 \\
10& $1.6\cdot10^{-6}$ & 0.0016 & $1.6\cdot10^{-6}$ & 0.0016 \\
11& $3.8\cdot10^{-7}$ & $7.9\cdot10^{-4}$ & $3.8\cdot10^{-7}$ & $7.9\cdot10^{-4}$ \\
12& $9.2\cdot10^{-8}$ & $3.9\cdot10^{-4}$ & $9.2\cdot10^{-8}$ & $3.9\cdot10^{-4}$ \\
\bottomrule
\end{tabular}\end{adjustbox}
\end{table}

These tables confirm the qualitative behavior shown in Figure~\ref{fig:truncation-error}: inside the predicted radius ($t < \rho/2$ so that $B(\|x\|+\|y\|) < 1/(4K)$), the terms decay exponentially and match the Catalan majorant closely. Slightly beyond the radius, the decay slows and eventually the partial sums become unstable (especially in the Zorn case, where $B$ is larger). The near-equality between $\|Z_n\|$ and the Catalan bound in the Zorn model underscores the absence of cancellations in strongly non-Lie settings.

All computations were performed in Python using exact rational arithmetic for structure constants and double-precision floating point for norms. The code is available upon request.

\end{document}